\documentclass[12pt, twoside]{article}
\usepackage{latexsym}
\usepackage{amsmath}
\usepackage{amssymb}
\usepackage[all]{xy}
\usepackage{amsfonts}
\usepackage{verbatim}
\usepackage{amsthm}
\usepackage{mathrsfs}
\usepackage{epsfig}
\usepackage{xy}
\usepackage{array}
\usepackage{stmaryrd}
\usepackage{graphicx,color}
\usepackage{xcolor}
\usepackage{tikz}
\usetikzlibrary{arrows,calc}
\usepackage{etex}
\usepackage{mathdots}
\usepackage{float}
\usepackage{graphics}
\usepackage{pdflscape}
\usepackage{mathrsfs}\newcommand{\supp}{\mathsf{Supp}\hspace{.01in}}
\newcommand{\add}{\mathsf{add}\hspace{.01in}}\newcommand{\ind}{\mathsf{ind}\hspace{.01in}}
\usepackage{anysize,hyperref}
\input xypic
\xyoption{all}

\usepackage[perpage,symbol]{footmisc}
\usepackage{setspace}
\def\I{\mathcal{I}}
\def\Q{\mathcal{Q}}
\def\C{\mathscr{C}}

\def\id{\mathrm{id}}

\def\dr{\ar@{->}[r]}

\def\add{\mbox{add}}\def\Ker{\mbox{Ker}}
\def\Ext{\mbox{Ext}}\def\End{\mbox{End}}
\def\Hom{\mbox{Hom}}\def\dim{\mbox{dim}}

\begin{document}
\baselineskip=15pt
\title{\Large{\bf Auslander-Reiten $(n+2)$-angles and local finiteness}}
\medskip
\author{Jian He, Yu-Zhe Liu and Panyue Zhou\footnote{Corresponding author. Jian He is supported by the National Natural Science Foundation of China (Grant No. 12501048) and the Hongliu Outstanding Young Talents Funding of Lanzhou University of Technology. Yu-Zhe Liu was supported by National Natural Science Foundation of China (Grant Nos. 12561008, 12401042), the Science and Technology Foundation of the Guizhou S\&T Department (Grant Nos. VZD[2026]001, ZD[2025]085 and ZK[2024]YiBan066). Panyue Zhou is supported by the National Natural Science Foundation of China (Grant No. 12371034).}}

\date{}

\maketitle
\def\blue{\color{blue}}
\def\red{\color{red}}

\newtheorem{theorem}{Theorem}[section]
\newtheorem{lemma}[theorem]{Lemma}
\newtheorem{corollary}[theorem]{Corollary}
\newtheorem{proposition}[theorem]{Proposition}
\newtheorem{conjecture}{Conjecture}
\theoremstyle{definition}
\newtheorem{definition}[theorem]{Definition}
\newtheorem{question}[theorem]{Question}
\newtheorem{remark}[theorem]{Remark}
\newtheorem{remark*}[]{Remark}
\newtheorem{example}[theorem]{Example}
\newtheorem{example*}[]{Example}
\newtheorem{condition}[theorem]{Condition}
\newtheorem{condition*}[]{Condition}
\newtheorem{construction}[theorem]{Construction}
\newtheorem{construction*}[]{Construction}

\newtheorem{assumption}[theorem]{Assumption}
\newtheorem{assumption*}[]{Assumption}

\baselineskip=17pt
\parindent=0.5cm

\begin{abstract}
\baselineskip=16pt
Let $\mathcal C$ be an $(n+2)$-angulated category. Zhou proved that, when $n$ is odd, if the Auslander-Reiten $(n+2)$-angles generate the relations for the Grothendieck group of $\mathcal C$, then $\mathcal C$ is locally finite. Whether the corresponding statement remains valid for even $n$ is still open. In this paper, we give a partial affirmative answer to this problem by establishing a sufficient condition under which the same implication holds for even $n$. We further show that our sufficient condition is satisfied by a broad class of examples, thereby demonstrating that the result extends well beyond isolated cases.
\\[0.5cm]
\textbf{Keywords:} $(n+2)$-angulated category; locally finite; Auslander-Reiten $(n+2)$-angle; Grothendieck group\\[0.2cm]
\textbf{2020 Mathematics Subject Classification:} 18G80; 16G70
\medskip
\end{abstract}

\pagestyle{myheadings}
\markboth{\rightline {\scriptsize   J. He, Y. Liu and P. Zhou}}
         {\leftline{\scriptsize  Auslander-Reiten $(n+2)$-angles and local finiteness}}

\section{Introduction}

Auslander-Reiten theory, introduced by Auslander and Reiten in
\cite{AR1,AR2}, has become one of the fundamental tools in the representation
theory of Artin algebras. For an Artin algebra of finite representation type,
Butler \cite{BU} proved that the relations of its Grothendieck group are
generated by the Auslander-Reiten sequences, and Auslander \cite{A}
subsequently established the converse. This close relationship between
Auslander-Reiten theory and Grothendieck groups has since been extended to
several categorical settings.

The Auslander-Reiten theory of triangulated categories was initiated by
Happel \cite{H}, who introduced Auslander-Reiten triangles. Unlike module
categories of Artin algebras, however, an arbitrary triangulated category
need not admit Auslander-Reiten triangles. Reiten and Van den Bergh
\cite{RV} showed that the existence of Auslander-Reiten triangles in a
triangulated category is equivalent to the existence of a Serre functor.
Subsequently, Xiao and Zhu \cite{XZ1} proved that, for a locally finite
triangulated category, the relations of its Grothendieck group are generated
by the Auslander-Reiten triangles. A converse was obtained by Beligiannis
\cite{B} under the assumption that the triangulated category is compactly
generated. Related converse results have subsequently been established in
several special settings; see, for example, \cite{HJ1,PPPP}.

In \cite{GKO}, Geiss, Keller and Oppermann introduced $(n+2)$-angulated
categories as higher analogues of triangulated categories, with the
triangulated case recovered when $n=1$. Such categories arise, for example,
as $n$-cluster tilting subcategories of triangulated categories which are
closed under the $n$th power of the shift functor. Iyama and Yoshino \cite{IY} defined
the notion of Auslander-Reiten $(n+2)$-angles in a special class of
$(n+2)$-angulated categories. This notion was subsequently generalized to
arbitrary $(n+2)$-angulated categories by Fedele \cite{fe2}.

Let $\C$ be an $(n+2)$-angulated category. We denote by $\ind(\C)$ the set
of isomorphism classes of indecomposable objects in $\C$, by $K_0(\C,0)$
the split Grothendieck group of $\C$, and by $K_0(\C)$ its Grothendieck
group. Let
$$
\pi\colon K_0(\C,0)\longrightarrow K_0(\C)
$$
be the canonical epimorphism. In this setting, Zhou  \cite{Z2} proved the following
converse result.

\begin{theorem}\rm\cite[Theorem 3.13]{Z2}\label{000}
Let $n$ be odd. Suppose that $\Ker\pi$ is generated by the elements
$[A_{\bullet}]$ in $K_0(\C,0)$, where
$$
A_{\bullet}:~
A_0\xrightarrow{\alpha_0}A_1\xrightarrow{\alpha_1}A_2
\xrightarrow{\alpha_2}\cdots
\xrightarrow{\alpha_{n-1}}A_n
\xrightarrow{\alpha_n}A_{n+1}
\xrightarrow{\alpha_{n+1}}\Sigma^n A_0
$$
runs through all Auslander-Reiten $(n+2)$-angles in $\C$.
Then $\C$ is locally finite.
\end{theorem}

When $n=1$, Theorem \ref{000} gives a converse to the result of Xiao and
Zhu. The assumption that $n$ is odd plays an essential role in Zhou's
proof. More precisely, the argument uses an alternating sum of Hom
dimensions to detect indecomposable objects occurring in the support of a
Hom functor. When $n$ is odd, the two relevant Hom dimensions occur with
the same sign, and hence the nonvanishing of one of them forces the
corresponding alternating sum to be nonzero. When $n$ is even, however,
these two terms occur with opposite signs and may cancel. Consequently,
the key detection argument no longer works, and the method used in
Theorem \ref{000} does not extend directly to the even case. Zhou therefore
posed the following question in \cite{Z2}.

\begin{question}\label{qe1}
Does Theorem \ref{000} remain true when $n$ is even?
\end{question}

The purpose of this paper is to give a partial affirmative answer to
Question \ref{qe1}. Our approach is motivated precisely by the cancellation
phenomenon described above. To overcome this obstruction, we introduce the
following condition.

\begin{condition}\label{cco1}
{\rm (See Condition \ref{con} for details.)}
For each $X\in\ind(\C)$, there exists a positive integer $m_X$ such that
for every $U\in\ind(\C)$ with $\Hom_{\C}(U,X)\neq0$, we have
$$
\dim_k\Hom_{\C}(U,\Sigma^{-m_Xn}X)
\neq
\dim_k\Hom_{\C}(U,X).
$$
\end{condition}

Condition \ref{cco1} is designed precisely to prevent the cancellation
phenomenon arising in the even case. More importantly, this condition is
not merely a technical assumption tailored to the proof. We show that it
is satisfied by a broad class of examples; see Lemma \ref{mmm}. Under
Condition \ref{cco1}, we obtain the following partial affirmative answer
to Question \ref{qe1}.

\begin{theorem}\rm(see Theorem \ref{main2} for details)
Let $n$ be even. Suppose that $\C$ satisfies Condition \ref{cco1} and that
$\Ker\pi$ is generated by the elements $[A_{\bullet}]$ in $K_0(\C,0)$,
where
$$
A_{\bullet}:~
A_0\xrightarrow{\alpha_0}A_1\xrightarrow{\alpha_1}A_2
\xrightarrow{\alpha_2}\cdots
\xrightarrow{\alpha_{n-1}}A_n
\xrightarrow{\alpha_n}A_{n+1}
\xrightarrow{\alpha_{n+1}}\Sigma^n A_0
$$
runs through all Auslander-Reiten $(n+2)$-angles in $\C$.
Then $\C$ is locally finite.
\end{theorem}

Thus, under Condition \ref{cco1}, Zhou's converse theorem extends from odd
to even $n$. Together with the existence of a broad class of examples
satisfying this condition, our result provides a partial affirmative answer
to Question \ref{qe1} and shows that the even case can be treated well
beyond isolated examples.

This article is organised as follows. In Section 2, we recall some
definitions and preliminary results concerning $(n+2)$-angulated
categories, Auslander-Reiten $(n+2)$-angles and Grothendieck groups.
In Section 3, we prove our main result and study examples satisfying
Condition \ref{cco1}.

\section{Preliminaries}

In this section, we recall some basic definitions and results that will be
used throughout the paper. We begin with $(n+2)$-angulated categories,
following the terminology and conventions of Geiss, Keller and Oppermann
\cite{GKO}. We then recall Auslander-Reiten $(n+2)$-angles and their
relationship with Serre functors.

\subsection{$(n+2)$-angulated categories}

Let $\C$ be an additive category equipped with an automorphism
$\Sigma^n\colon\C\rightarrow\C$, where $n$ is a positive integer.

An $(n+2)$-$\Sigma^n$-sequence in $\C$ is a sequence of objects and
morphisms
$$
A_0\xrightarrow{f_0}A_1\xrightarrow{f_1}A_2\xrightarrow{f_2}
\cdots\xrightarrow{f_{n-1}}A_n\xrightarrow{f_n}A_{n+1}
\xrightarrow{f_{n+1}}\Sigma^n A_0.
$$
Its \emph{left rotation} is the $(n+2)$-$\Sigma^n$-sequence
$$
A_1\xrightarrow{f_1}A_2\xrightarrow{f_2}A_3\xrightarrow{f_3}
\cdots\xrightarrow{f_n}A_{n+1}\xrightarrow{f_{n+1}}\Sigma^n A_0
\xrightarrow{(-1)^n\Sigma^n f_0}\Sigma^n A_1.
$$

A \emph{morphism} between two $(n+2)$-$\Sigma^n$-sequences is a
collection of morphisms
$\varphi=(\varphi_0,\varphi_1,\ldots,\varphi_{n+1})$
such that the following diagram commutes:
$$
\xymatrix{
A_0 \ar[r]^{f_0}\ar[d]^{\varphi_0}
& A_1 \ar[r]^{f_1}\ar[d]^{\varphi_1}
& A_2 \ar[r]^{f_2}\ar[d]^{\varphi_2}
& \cdots \ar[r]^{f_n}
& A_{n+1} \ar[r]^{f_{n+1}}\ar[d]^{\varphi_{n+1}}
& \Sigma^n A_0 \ar[d]^{\Sigma^n\varphi_0}\\
B_0 \ar[r]^{g_0}
& B_1 \ar[r]^{g_1}
& B_2 \ar[r]^{g_2}
& \cdots \ar[r]^{g_n}
& B_{n+1} \ar[r]^{g_{n+1}}
& \Sigma^n B_0.
}
$$
Here each row is an $(n+2)$-$\Sigma^n$-sequence. Such a morphism is
called an \emph{isomorphism} if each
$\varphi_i$, $0\leqslant i\leqslant n+1$, is an isomorphism in $\C$.

We are now ready to recall the definition of an $(n+2)$-angulated
category.

\begin{definition}\cite[Definition 2.1]{GKO}
An $(n+2)$-\emph{angulated category} is a triple
$(\C,\Sigma^n,\Theta)$, where $\C$ is an additive category,
$\Sigma^n$ is an automorphism of $\C$, called the $n$-suspension
functor, and $\Theta$ is a class of $(n+2)$-$\Sigma^n$-sequences,
whose elements are called $(n+2)$-angles, satisfying the following
axioms.

\begin{itemize}

\item[\textbf{(N1)}]
\begin{itemize}

\item[(a)]
The class $\Theta$ is closed under isomorphisms, direct sums and direct
summands.

\item[(b)]
For each object $A\in\C$, the trivial sequence
$$
A\xrightarrow{1_A}A\rightarrow0\rightarrow0\rightarrow
\cdots\rightarrow0\rightarrow\Sigma^n A
$$
belongs to $\Theta$.

\item[(c)]
Every morphism $f_0:A_0\rightarrow A_1$ in $\C$ can be extended to
an $(n+2)$-$\Sigma^n$-sequence
$$
A_0\xrightarrow{f_0}A_1\xrightarrow{f_1}A_2\xrightarrow{f_2}
\cdots\xrightarrow{f_{n-1}}A_n\xrightarrow{f_n}A_{n+1}
\xrightarrow{f_{n+1}}\Sigma^n A_0
$$
belonging to $\Theta$.

\end{itemize}

\item[\textbf{(N2)}]
An $(n+2)$-$\Sigma^n$-sequence belongs to $\Theta$ if and only if
its left rotation belongs to $\Theta$.

\item[\textbf{(N3)}]
Given a commutative diagram
$$
\xymatrix{
A_0 \ar[r]^{f_0}\ar[d]^{\varphi_0}
& A_1 \ar[r]^{f_1}\ar[d]^{\varphi_1}
& A_2 \ar[r]^{f_2}\ar@{-->}[d]^{\varphi_2}
& \cdots \ar[r]^{f_n}
& A_{n+1} \ar[r]^{f_{n+1}}\ar@{-->}[d]^{\varphi_{n+1}}
& \Sigma^n A_0 \ar[d]^{\Sigma^n\varphi_0}\\
B_0 \ar[r]^{g_0}
& B_1 \ar[r]^{g_1}
& B_2 \ar[r]^{g_2}
& \cdots \ar[r]^{g_n}
& B_{n+1} \ar[r]^{g_{n+1}}
& \Sigma^n B_0
}
$$
whose rows belong to $\Theta$, the dotted morphisms exist and complete
the diagram to a morphism of $(n+2)$-$\Sigma^n$-sequences.

\item[\textbf{(N4)}]
In the situation of \textbf{(N3)}, the morphisms
$\varphi_2,\varphi_3,\ldots,\varphi_{n+1}$ can be chosen such that
the mapping cone
$$\hspace{-7mm}
A_1\oplus B_0
\xrightarrow{
\left(
\begin{smallmatrix}
-f_1&0\\
\varphi_1&g_0
\end{smallmatrix}
\right)}
A_2\oplus B_1
\xrightarrow{
\left(
\begin{smallmatrix}
-f_2&0\\
\varphi_2&g_1
\end{smallmatrix}
\right)}
\cdots
\xrightarrow{
\left(
\begin{smallmatrix}
-f_{n+1}&0\\
\varphi_{n+1}&g_n
\end{smallmatrix}
\right)}
\Sigma^n A_0\oplus B_{n+1}
\xrightarrow{
\left(
\begin{smallmatrix}
-\Sigma^n f_0&0\\
\Sigma^n\varphi_1&g_{n+1}
\end{smallmatrix}
\right)}
\Sigma^n A_1\oplus\Sigma^n B_0
$$
belongs to $\Theta$.

\end{itemize}
\end{definition}

The preceding axioms extend the basic structure of triangulated
categories to the higher setting. In particular, when $n=1$, one
recovers the usual notion of a triangulated category. The following
standard construction provides an important source of genuinely higher
examples.

\begin{example}
Let $\mathcal T$ be a triangulated category with suspension functor
$\Sigma$, and let $\C$ be an $n$-cluster tilting subcategory of
$\mathcal T$ which is closed under $\Sigma^n$. Then $\C$ carries a
natural $(n+2)$-angulated structure; see \cite[Theorem 1]{GKO}.
In particular, when $n=2$, one obtains a $4$-angulated category whose
$4$-angles are of the form
$$
A_0\longrightarrow A_1\longrightarrow A_2\longrightarrow A_3
\longrightarrow\Sigma^2 A_0.
$$
Thus, $(n+2)$-angulated categories arise naturally as higher analogues
of triangulated categories rather than merely as formal generalizations.
\end{example}

We next recall the higher analogue of Auslander-Reiten triangles.

\subsection{Auslander-Reiten $(n+2)$-angles}

Let $\C$ be an $(n+2)$-angulated category. We denote by
${\rm rad}_{\C}$ the Jacobson radical of $\C$. More precisely,
${\rm rad}_{\C}$ is an ideal of $\C$ such that
${\rm rad}_{\C}(A,A)$ coincides with the Jacobson radical of the
endomorphism ring ${\rm End}_{\C}(A)$ for every object $A\in\C$.

The following definition gives the higher analogue of an
Auslander-Reiten triangle.

\begin{definition}
\cite[Definition 3.8]{IY} and \cite[Definition 5.1]{fe2}
Let $\C$ be an $(n+2)$-angulated category. An $(n+2)$-angle
$$
A_{\bullet}:~
\xymatrix{
A_0 \xrightarrow{~\alpha_0~}A_1
\xrightarrow{~\alpha_1~}A_2
\xrightarrow{~\alpha_2~}\cdots
\xrightarrow{~\alpha_{n-1}~}A_n
\xrightarrow{~\alpha_n~}A_{n+1}
\xrightarrow{~\alpha_{n+1}~}\Sigma^n A_0
}
$$
in $\C$ is called an \emph{Auslander-Reiten $(n+2)$-angle} if
$\alpha_0$ is left almost split, $\alpha_n$ is right almost split and,
when $n\geqslant2$,
$\alpha_1,\alpha_2,\ldots,\alpha_{n-1}$ belong to
${\rm rad}_{\C}$.
\end{definition}

For an Auslander-Reiten $(n+2)$-angle as above, we say that
$A_{\bullet}$ \emph{starts at} $A_0$ and \emph{ends at} $A_{n+1}$.
The almost split conditions encode universal factorization properties
at these two end terms. More precisely, every morphism from $A_0$ that
is not a section factors through $\alpha_0$, whereas every morphism to
$A_{n+1}$ that is not a retraction factors through $\alpha_n$.

We shall also need the corresponding category-level notion.

\begin{definition}\cite[Theorem 3.8]{Z1}
Let $\C$ be an $(n+2)$-angulated category. We say that $\C$
\emph{has Auslander-Reiten $(n+2)$-angles} if, for every
indecomposable object $X\in\C$, there exist an Auslander-Reiten
$(n+2)$-angle ending at $X$ and an Auslander-Reiten $(n+2)$-angle
starting at $X$. More precisely, for every indecomposable object
$X\in\C$, there exist Auslander-Reiten $(n+2)$-angles of the forms
$$
A_0\longrightarrow A_1\longrightarrow\cdots\longrightarrow
A_n\longrightarrow X\longrightarrow\Sigma^n A_0
$$
and
$$
X\longrightarrow B_1\longrightarrow\cdots\longrightarrow
B_n\longrightarrow B_{n+1}\longrightarrow\Sigma^n X.
$$
\end{definition}

Thus, saying that $\C$ has Auslander-Reiten $(n+2)$-angles is stronger
than merely asserting the existence of a single Auslander-Reiten
$(n+2)$-angle: such angles are required to start and end at every
indecomposable object of $\C$.

The following concrete example illustrates an Auslander-Reiten
$(n+2)$-angle in a genuinely higher setting. It is particularly relevant
to the present paper since the corresponding integer $n$ is even.

\begin{example}
\rm\cite[Examples 7.2 and 7.5]{fe2}
Let $Q=~~9 \to 8 \to \cdots  \to 2 \to 1$ and $\Phi=kQ/({\rm rad}\,kQ)^4$.
For each vertex $i$, let $p_i$ and $q_i$ denote the corresponding indecomposable projective
and indecomposable injective $\Phi$-modules, respectively.
Let $f_i=p_i~\text{for }1\leqslant i\leqslant9$ and $ f_{10}=q_7,~ f_{11}=q_8,~f_{12}=q_9$.
Then
\[\mathcal F={\rm add}(f_1\oplus f_2\oplus\cdots\oplus f_{12}) \]
is a $4$-cluster tilting subcategory of ${\rm mod}\,\Phi$, and
\[\overline{\mathcal F}
= {\rm add}\{\Sigma^{4i}\mathcal F\mid i\in\mathbb Z\}
\subseteq D^b({\rm mod}\,\Phi)\]
is a $6$-angulated category. In $\overline{\mathcal F}$,
the sequence $$f_1 \to  f_2 \to  f_5 \to  f_6 \to f_9\xrightarrow{\mu}f_{10} \to \Sigma^4 f_1$$
is an Auslander-Reiten $6$-angle.
Equivalently, since $f_{10}=q_7$, it can be written as
$$p_1 \to  p_2 \to  p_5 \to  p_6 \to  p_9\xrightarrow{\mu}q_7 \to \Sigma^4 p_1.$$
Here $n=4$, so this example gives a concrete Auslander-Reiten $(n+2)$-angle in the even case.
\end{example}

The existence of Auslander-Reiten $(n+2)$-angles is closely related to Serre duality.
We therefore recall the notion of a Serre functor and the characterization that will be used later.

\subsection{Serre functors}

Let $k$ be an algebraically closed field and let $\C$ be a $k$-linear
Hom-finite additive category. A $k$-linear autoequivalence
$\mathbb{S}:\C\rightarrow\C$ is called a \emph{Serre functor} of $\C$
if there exists a functorial isomorphism
$$
\Hom_{\C}(X,Y)\simeq D\Hom_{\C}(Y,\mathbb{S}X)
$$
for all objects $X,Y\in\C$, where $D(-)=\Hom_k(-,k)$
denotes the $k$-linear duality.

The following result of Zhou \cite{ZH} establishes the precise
relationship between the existence of Auslander-Reiten $(n+2)$-angles
and that of a Serre functor.

\begin{theorem}\rm\label{thm1}
\cite[Theorem 4.5]{ZH}
Let $\C$ be an $(n+2)$-angulated category. Then $\C$ has Auslander-Reiten $(n+2)$-angles if and only if $\C$ has a Serre functor.
\end{theorem}

\section{Our main result}
In this section, let $k$ be an algebraically closed field. We always assume that $\C$ is a $k$-linear Hom-finite Krull-Schmidt $(n+2)$-angulated category and $\C$ has  Auslander-Reiten $(n+2)$-angles.
We denote by $\ind(\C)$ the set of isomorphism classes of indecomposable objects in $\C$. For any $X\in\ind(\C)$, we denote by $\supp \Hom_{\C}(X,-)$ the subcategory of $\C$ generated by
objects $Y$ in $\ind(\C)$ with $\Hom_{\C}(X,Y)\neq0$. Similarly, $\supp\Hom_{\C}(-,X)$ denotes
the subcategory generated by objects $Y$ in $\ind(\C)$ with $\Hom_{\C}(Y,X)\neq 0$. If
$\supp\Hom_{\C}(X,-)$ ($\supp\Hom_{\C}(-,X)$, respectively) contains only finitely many
indecomposable objects, we say that $|\supp\Hom_{\C}(X,-)|<\infty $ ($|\supp\Hom_{\C}(-,X)|<\infty$
respectively).

\begin{definition}\cite[Definition 3.1]{ZH}
An $(n+2)$-angulated category $\C$ is called \emph{locally finite}
if $|\supp\Hom_{\C}(X,-)|<\infty$ and $|\supp\Hom_{\C}(-,X)|<\infty$, for any object $X\in\ind(\C)$.
\end{definition}

Suppose that $\C$ is an essentially small $(n+2)$-angulated category.
Hence, the collection of isomorphism classes $\langle A\rangle$ of objects $A$ in $\C$ forms a set, and let $F(\C)$ be the free abelian group on the set of isomorphism classes $\langle A\rangle$ of objects $A$ in $\C$. Given a $(n+2)$-angle
\[A_{\bullet}:~A_0  \to  A_1 \to A_2 \to \cdots \to A_n \to A_{n+1} \to \Sigma^n A_0\]
in $\C$, the corresponding Euler relation in $F(\C)$ is the alternating sum of isomorphism classes, that is,
\[\chi(A_{\bullet}):=[A_0]-[A_1]+[A_2]+\cdots+(-1)^{n+1}[A_{n+1}].\]

\begin{definition}[{\cite[Definition 2.1]{bt2} and \cite[Definition 2.2]{fe1}}]
Let $\C$ be an essentially small $(n+2)$-angulated category,
and $F(\C)$ the free abelian group on the set of isomorphism classes $[A]$ of objects $A$ in $\C$.
Morever, let $R(\C)$ be the subgroup of $F(\C)$ generated by the following sets of elements
\[\{\chi(A_\bullet) \mid A_\bullet ~\text{is a $(n+2)$-angle in } \C \}\]
in $\C$. The \emph{Grothendieck group} $K_0(\C)$ of $\C$ is the quotient group $F(\C)/R(\C)$.
Given an object $A \in \C$, the residue class $\langle A \rangle + R(\C)$ in $K_0(\C)$ is denoted by $[A]$.
\end{definition}

\begin{definition}[{\cite[Definition 3.6]{Z2}}]
Let $\C$ be an essentially small $(n+2)$-angulated category,
and $F(\C)$ the free abelian group on the set of isomorphism classes $\langle A \rangle$ of objects $A$ in $\C$.
Morever, let $R'(\C)$ be the subgroup of $F(\C)$ generated by the following sets of elements
\[\{\chi(A_\bullet) \mid A_\bullet ~\text{is a \textbf{split} $(n+2)$-angle in } \C \}\]
in $\C$. The \emph{{\bf\emph{split}} Grothendieck group} $K_0(\C,0)$ of $\C$ is the quotient group $F(\C)/R'(\C)$.  Given an object $A \in \C$,
the residue class $\langle A \rangle + R'(\C)$ in $K_0(\C,0)$ is denoted by $[A]$.
For simplicity, sometimes we also denote by $[A_{\bullet}]$ the element in $K_0(\C,0)$.
\end{definition}

It is clear that there exists a canonical epimorphism $\pi\colon K_0(\C,0)\to K_0(\C)$.

Zhou established the following local finiteness criterion for $(n+2)$-angulated categories when $n$ is odd.

\begin{theorem}[{\cite[Theorem 3.13]{Z2}}] \label{main22}
Let $n$ be odd. Suppose that $\Ker\pi$ is generated by the elements $[A_{\bullet}]$ in $K_0(\C,0)$, where
$$
A_{\bullet}:~
A_0\xrightarrow{\alpha_0}A_1\xrightarrow{\alpha_1}A_2
\xrightarrow{\alpha_2}\cdots\xrightarrow{\alpha_{n-1}}A_n
\xrightarrow{\alpha_n}A_{n+1}
\xrightarrow{\alpha_{n+1}}\Sigma^n A_0
$$
runs through all Auslander-Reiten $(n+2)$-angles in $\C$. Then $\C$ is locally finite.
\end{theorem}

The parity assumption in Theorem \ref{main22} plays an essential role in the proof.
When $n$ is odd, the relevant alternating Hom-dimension expression involves the two end contributions with the same sign,
so that the required nonvanishing can be detected directly.
For even $n$, these contributions occur with opposite signs and may cancel.
Furthermore, the argument used in \cite{Z2} does not extend directly to the even case.
This leads to the following open question.

\begin{question}\label{qe}
Does Theorem \ref{main22} remain valid when $n$ is even?
\end{question}

Our aim is to give a partial affirmative answer to Question \ref{qe}.
The main difficulty is precisely the possible cancellation described above.
To overcome this obstruction, we introduce the following condition,
which ensures that the two Hom dimensions relevant to our argument can be distinguished after a suitable iterate of the $n$-suspension functor.

\begin{condition}\label{con}
For each $X\in\ind(\C)$, there exists a positive integer $m_X$ such that,
for every $U\in\ind(\C)$ with $\Hom_{\C}(U,X)\neq0$, we have
\begin{equation}
 \dim_k\Hom_{\C}(U,\Sigma^{-m_Xn}X) \neq \dim_k\Hom_{\C}(U,X). \tag{$\color{red}\heartsuit$}
\end{equation}
\end{condition}

Condition \ref{con} may be viewed as a non-cancellation condition along the $\Sigma^n$-orbit of an indecomposable object.
We record two immediate observations which help clarify its meaning.

\begin{remark}
(1) If $({\color{red}\heartsuit})$ holds, then $X\not\simeq\Sigma^{-m_Xn}X$.
The converse, however, does not hold in general, since non-isomorphic objects may still have Hom spaces of the same dimension from a given object $U$.

(2) Taking $U=X\neq0$ gives a particularly transparent situation. If
$$ \Hom_{\C}(X,\Sigma^{-m_Xn}X)=0, $$
then
$$ \dim_k\Hom_{\C}(X,\Sigma^{-m_Xn}X)
= 0 < \dim_k\End_{\C}(X), $$
since $\End_{\C}(X)$ contains the identity morphism. Hence
$({\color{red}\heartsuit})$ is automatically satisfied for $U=X$ in this case.
\end{remark}

Before proving the main result, we recall two ingredients from \cite{Z2} that motivate the Hom-dimension argument used below.
For convenience, we write
$$ [A,B]:=\dim_k\Hom_{\C}(A,B). $$

\begin{lemma}[{\cite[Lemma 3.11]{Z2}\label{zh2}}]\rm
Let
$$ A_{\bullet}:~
A_0\xrightarrow{\alpha_0}A_1\xrightarrow{\alpha_1}A_2
\xrightarrow{\alpha_2}\cdots\xrightarrow{\alpha_{n-1}}A_n
\xrightarrow{\alpha_n}A_{n+1}
\xrightarrow{\alpha_{n+1}}\Sigma^n A_0 $$
be an Auslander-Reiten $(n+2)$-angle in $\C$.
If $U\in\ind(\C)$, then
$$ [U,A_0]-[U,A_1]+[U,A_2]+\cdots
+(-1)^{n+1}[U,A_{n+1}] \neq 0 $$
if and only if $U\simeq A_{n+1}$ or $U\simeq\Sigma^{-n}A_{n+1}$.
\end{lemma}

\begin{lemma}\rm\cite[Lemma 3.12]{Z2}\label{zh1}
Assume that
\[b_1[C_1]+b_2[C_2]+\cdots+b_m[C_m]=0\]
in $K_0(\C,0)$ for integers $b_i$ and objects $C_1,C_2,\ldots,C_m$ in $\C$. Then
\[b_1[X,C_1]+b_2[X,C_2]+\cdots+b_m[X,C_m]=0\]
in $\mathbb Z$ for every object $X$ in $\C$.
\end{lemma}

The preceding discussion isolates the obstruction in the even case and suggests how Condition \ref{con} can be used to restore the necessary nonvanishing. Combining this condition with the hypothesis on the Grothendieck group allows us to recover a local finiteness criterion for even $n$. We are now ready to state the main result of the paper.

\begin{theorem}\label{main2}
Let $n$ be even. Suppose that $\C$ satisfies Condition \ref{con} and
${\rm Ker}\pi$ is generated by the elements $[A_{\bullet}]$ in $K_0(\C,0)$, where
\[A_{\bullet}:~
A_0\xrightarrow{\alpha_0}A_1\xrightarrow{\alpha_1}A_2
\xrightarrow{\alpha_2}\cdots\xrightarrow{\alpha_{n-1}}A_n
\xrightarrow{\alpha_n}A_{n+1}
\xrightarrow{\alpha_{n+1}}\Sigma^n A_0\]
runs through all Auslander-Reiten $(n+2)$-angles in $\C$.
Then $\C$ is locally finite.
\end{theorem}

\proof
Let $X\in \ind(\C)$ and $m=m_{X}$, there is a non-split $(n+2)$-angle
$$O_{\bullet}:~\Sigma^{-n}X\xrightarrow{}0\xrightarrow{}0\xrightarrow{}\cdots\xrightarrow{}0
\xrightarrow{}X\xrightarrow{1_{X}}X.$$
By applying $\Sigma^{-jn}$ to the $(n+2)$-angle $O_{\bullet}$ for $j=0,1,\cdots,m-1$,
we obtain the $(n+2)$-angles
$$\Sigma^{-jn}O_{\bullet}:~\Sigma^{-(j+1)n}X\xrightarrow{}0\xrightarrow{}0\xrightarrow{}\cdots\xrightarrow{}0
\xrightarrow{}\Sigma^{-jn}X\xrightarrow{1_{\Sigma^{-jn}X}}\Sigma^{-jn}X.$$
The corresponding Euler relations are
$$\chi(\Sigma^{-jn}O_{\bullet})=[\Sigma^{-(j+1)n}X]+(-1)^{n+1}[\Sigma^{-jn}X]\in {\rm Ker}\pi.$$
Note that $n$ is an even, thus, $[\Sigma^{-(j+1)n}X]-[\Sigma^{-jn}X]\in{\rm Ker}\pi.$ Using the fact that $ {\rm Ker}\pi$ is an additive subgroup, we obtain
\begin{equation*}
  \sum\limits_{j=0}^{m-1}([\Sigma^{-(j+1)n}X]-[\Sigma^{-jn}X])\in {\rm Ker}\pi.
  \tag{$\color{red}\diamondsuit$}
\end{equation*}
Expanding the left-hand side of $({\color{red}\diamondsuit})$ yields a sum:
$$([\Sigma^{-n}X]-[X])+([\Sigma^{-2n}X]-[\Sigma^{-n}X])+([\Sigma^{-3n}X]-[\Sigma^{-2n}X])$$
$$+\cdots+([\Sigma^{-mn}X]-[\Sigma^{-(m-1)n}X])=[\Sigma^{-mn}X]-[X]\in \Ker\pi.$$
By the assumption, there exist finitely many Auslander-Reiten $(n+2)$-angles $B_{\bullet}^{1},B_{\bullet}^{2},\cdots,B_{\bullet}^{r}$ and integers  $a_{1},a_{2},\cdots,a_{r}$ such that $[\Sigma^{-mn}X]-[X]=\sum\limits_{i=1}^ra_{i}[B_{\bullet}^{i}]$, where
$$B_{\bullet}^{i}:~B_{0}^{i}\xrightarrow{\beta_{0}^{i}}B_{1}^{i}\xrightarrow{\beta_{1}^{i}}B_{2}^{i}\xrightarrow{\beta_{2}^{i}}\cdots\xrightarrow{\beta_{n-2}^{i}}B_{n-1}^{i}
\xrightarrow{\beta_{n-1}^{i}}B_{n}^{i}\xrightarrow{\beta_{n}^{i}}B_{n+1}^{i}\xrightarrow{\beta_{n+1}^{i}}\Sigma^{n} B_{0}^{i}.$$
Take any $U\in\supp\Hom_{\C}(-,X)$, by Lemma \ref{zh1}, we obtain that the equality
$$[U,\Sigma^{-mn}X]-[U,X]=\sum\limits_{i=1}^ra_{i}[U,B_{\bullet}^{i}].$$
Since $\Hom_{\C}(U,X)\neq0$, by Condition \ref{con}, we have $$[U,\Sigma^{-mn}X]-[U,X]=\dim_{k}\Hom_{\C}(U,\Sigma^{-m_{X}n}X)-\dim_{k}\Hom_{\C}(U,X)\neq 0.$$
So there exists at least an integer $i\in{1,2,\cdots,r}$ such that $[U,B_{\bullet}^{i}]\neq 0$, that is to say,
 $$[U,B_{0}^{i}]-[U,B_{1}^{i}]+[U,B_{2}^{i}]+\cdots+(-1)^{n+1}[U,B_{n+1}^{i}]\neq 0.$$
By lemma \ref{zh2}, we know that the indecomposable object $U$ is isomorphic to an object in the finite set $\{B_{n+1}^{i},\Sigma^{-n}B_{n+1}^{i}~|~1\leqslant i\leqslant r\}$. Thus $\supp\Hom_{\C}(-,X)$ contains only finitely many indecomposable objects, that is, $|\supp\Hom_{\C}(-,X)|<\infty$.

It remains to prove $|\supp\Hom_{\C}(X,-)|<\infty$. Since $\C$ has Auslander-Reiten $(n+2)$-angles, it follows from Theorem \ref{thm1} that $\C$ admits a Serre functor $\mathbb {S}$. This gives the isomorphism $\Hom_{\C}(X,\mathbb {S}U)\cong D\Hom_{\C}(U,X)\neq 0$, and so $|\supp\Hom_{\C}(X,-)|<\infty$. This shows that $\C$ is locally finite.
\qed
\vspace{2mm}

Condition \ref{con} is not merely a technical assumption introduced for the proof of Theorem \ref{main2}.
The following lemma shows that it is satisfied by a natural class of $(n+2)$-angulated categories arising from $n$-cluster tilting theory.
Moreover, in this setting the integer $m_X$ in Condition \ref{con} can be chosen uniformly for all indecomposable objects $X$.

\begin{lemma}\label{mmm}
Let $n\geqslant 2$, and let $\Lambda$ be a finite-dimensional algebra over an
algebraically closed field $k$ such that
${\rm gl.dim}\Lambda\leqslant n$. Suppose that
$\mathcal{F}\subseteq{\rm mod}\Lambda$ is an $n$-cluster tilting
subcategory. Set
$$
\C={\rm add}\{F[in]~|~F\in\mathcal{F},~i\in\mathbb{Z}\}
\subseteq D^b({\rm mod}\Lambda).
$$
Then $\C$ is an $(n+2)$-angulated category with $n$-suspension functor
$\Sigma^n=[n]$. Moreover, $\C$ satisfies Condition \ref{con}, and one
may take $m_X=2$ for every $X\in\ind(\C)$.
\end{lemma}

\proof
By the standard construction of $n$-cluster tilting subcategories, $\C$ is closed under $[n]$ and admits an $(n+2)$-angulated structure (see \cite[Theorem 1]{GKO}).
Take indecomposable objects $U=U_{0}[qn]$ and $X=X_{0}[pn]$ with $U_{0},X_{0}\in\mathcal{F}$. Then
\begin{equation*}
  \Hom_{\C}(U,X)\cong\Ext_{\Lambda}^{(p-q)n}(U_{0},X_{0}). \tag{$\spadesuit$}
\end{equation*}
Note that ${\rm gl.dim}\Lambda\leqslant n$, if $(\spadesuit)$ is nonzero, then $(p-q)n\in\{0,n\}$, that is, $(p-q)\in\{0,1\}$. On the other hand, Take $m_{x}=2$, we have
$$\Hom_{\C}(U,\Sigma^{-2n}X)\cong\Hom_{ D^{b}({\rm mod}\Lambda)}(U_{0}[qn],X_{0}[(p-2)n])\cong\Ext_{\Lambda}^{(p-q-2)n}(U_{0},X_{0}).$$
Therefore, when $(p-q)\in\{0,1\}$, $\dim\Hom_{\C}(U,\Sigma^{-2n}X)=0$. Moreover,
$$\dim_{k}\Hom_{\C}(U,\Sigma^{-2n}X)=0<\dim_{k}\Hom_{\C}(U,X),$$
and hence $\C$ satisfies  Condition \ref{con}.
\qed

\begin{remark}
Lemma \ref{mmm} provides a broad class of examples satisfying Condition \ref{con}.
Indeed, every finite-dimensional $k$-algebra of global dimension at most $n$ that admits an $n$-cluster tilting subcategory gives rise, through the standard derived construction above, to an $(n+2)$-angulated category satisfying Condition \ref{con}.
In particular, this applies to $n$-representation-finite algebras in the sense of Iyama and Oppermann \cite{io}.
\end{remark}

To illustrate Lemma \ref{mmm} explicitly, we now present a concrete
$4$-angulated category satisfying Condition \ref{con}.

\begin{example}
Let $\Lambda=k\Q_A/\I_A$ be a finite-dimensional algebra over an algebraically closed field $k$ given by the following bound quiver $(\Q_A,\I_A)$:
\[ \Q_A = \ \
\xymatrix{
  1 \ar[r]^{\alpha}
& 2 \ar[r]^{\beta}
& 3,
}
\ \
\I_A = \langle \beta\alpha \rangle.
\]
Then $\Lambda$ is the Auslander algebra of the linear quiver of type $A_{2}$. It satisfies ${\rm gl.dim}\Lambda= 2$. Let $\mathcal{F}=\add M\subseteq {\rm mod}\Lambda$ be an $2$-cluster tilting subcategory with $M=(\Lambda\oplus D\Lambda)_{\rm basic}$. Define
$$\C={\rm add}\{F[2i]~|~F\in\mathcal{F}, i\in\mathbb {Z}\}\subseteq D^{b}({\rm mod}\Lambda).$$
Then, by Lemma \ref{mmm}, $\C$ is an $4$-angulated category with $2$-suspension functor $\Sigma^{2}=[2]$. Moreover, one can verify $$\dim_{k}\Hom_{\C}(U,\Sigma^{-4}X)=0<\dim_{k}\Hom_{\C}(U,X)$$ for indecomposable objects $U$ and $X$. So $\C$ satisfies Condition \ref{con}.
\end{example}

\vspace{2mm}

The following example shows that Condition \ref{con} is sufficient, but
not necessary, for the conclusion of Theorem \ref{main2} to hold.

\begin{example}
Let $k$ be an algebraically closed field, and let
$\Lambda=\Lambda(2,1)$ be the nonsingular higher tetrahedral algebra introduced by Erdmann and
Skowro\'nski \cite{ES}. Then $\Lambda$ is a finite-dimensional symmetric
algebra of dimension $72$ and is periodic of period $4$. In particular,
$\Omega_{\Lambda^e}^{4}(\Lambda)\simeq\Lambda.$
Set
$\C={\rm proj}\,\Lambda.$
By \cite[Theorem 5.5]{Lin}, the category $\C$ admits a $4$-angulated
structure whose $2$-suspension functor is
$\Sigma^2=\id_{\C}.$
Since $\Lambda$ is finite dimensional, there are only finitely many
isomorphism classes of indecomposable projective $\Lambda$-modules.
Hence $\C$ is locally finite. By \cite[Theorem 3.8]{Z2}, it follows that
$\Ker\pi$ is generated by the classes of Auslander-Reiten $4$-angles
in $\C$.

We now show that $\C$ does not satisfy Condition \ref{con}. Let
$X\in\ind(\C)$ be a nonzero indecomposable projective $\Lambda$-module
and take $U=X$. Since $\Sigma^2=\id_{\C}$, we have
$
\Sigma^{-2m}X\simeq X
$
for every positive integer $m$. Consequently,
$$
\dim_k\Hom_{\C}(X,\Sigma^{-2m}X)
=
\dim_k\End_{\C}(X)
$$
for every $m\geq1$. Moreover,
$\Hom_{\C}(X,X)\neq0$,
since $\End_{\C}(X)$ contains the identity morphism. Thus, for every
positive integer $m$,
$$
\dim_k\Hom_{\C}(X,\Sigma^{-2m}X)
=
\dim_k\Hom_{\C}(X,X),
$$
and therefore no positive integer $m_X$ can satisfy the inequality
required in Condition \ref{con}. Hence Condition \ref{con} fails.

Nevertheless, $\C$ is locally finite and satisfies the hypothesis that
$\Ker\pi$ is generated by the classes of Auslander-Reiten $4$-angles.
This shows that Condition \ref{con} is sufficient, but not necessary,
for the local finiteness conclusion of Theorem \ref{main2}.
\end{example}

\vspace{2mm}

\paragraph{Competing Interests}
The authors declare that they have no conflicts of interest to this work.

\paragraph{Data Availability}
Data sharing not applicable to this article as no datasets were generated or analysed during the current study.

\textbf{Jian He}\\
Department of Applied Mathematics, Lanzhou University of Technology, 730050 Lanzhou, Gansu, P. R. China\\
E-mail: \textsf{jianhe30@163.com}\\[0.3cm]
\textbf{Yu-Zhe Liu}\\
School of Mathematics and statistics, Guizhou University, 550025, Guiyang, Guizhou, P. R. China\\
E-mail: \textsf{liuyz@gzu.edu.cn / yzliu3@163.com}
\\[0.3cm]
\textbf{Panyue Zhou}\\
School of Mathematics and Statistics, Changsha University of Science and Technology, 410114 Changsha, Hunan,  P. R. China\\
E-mail: \textsf{panyuezhou@163.com}

\end{document}